\documentclass[12p, reqno]{amsart}
\usepackage{amsfonts}
\usepackage{amsmath}
\usepackage{amssymb}

\theoremstyle{plain}

\newtheorem{thm}{Theorem}[section]
\newtheorem{cor}{Corollary}
\newtheorem{lem}{Lemma}[section]
\newtheorem{rem}{Remark}
\newtheorem{prop}{Proposition}[section]

\numberwithin{equation}{section}

\DeclareMathOperator{\re}{Re}
\DeclareMathOperator{\im}{Im}
\usepackage{graphicx}
\graphicspath{%
    {converted_graphics/}
    {/}
}
\begin{document}
\def\R{{\mathbb R}}
\def\C{{\mathbb C}}
\def\N{{\mathbb N}}
\def\DD{{\mathbb D}}
\def\S{{\mathbb S}}
\def\rr{{\cal R}}
\def\e{\emptyset}
\def\dQ{\partial Q}
\def\dk{\partial K}
\def\endofproof{{\rule{6pt}{6pt}}}
\def\di{\displaystyle}
\def\dist{\mbox{\rm dist}}
\def\u-{\overline{u}}
\def\du{\frac{\partial}{\partial u}}
\def\dv{\frac{\partial}{\partial v}}
\def\dt{\frac{d}{d t}}
\def\dx{\frac{\partial}{\partial x}}
\def\con{\mbox{\rm const }}
\def\Box{\spadesuit}
\def\ii{{\bf i}}
\def\curl{{\rm curl}\,}
\def\dive{{\rm div}\,}
\def\grad{{\rm grad}\,}
\def\dist{\mbox{\rm dist}}
\def\pr{\mbox{\rm pr}}
\def\pp{{\cal P}}
\def\supp{\mbox{\rm supp}}
\def\Arg{\mbox{\rm Arg}}
\def\In{\mbox{\rm Int}}
\def\Re{\mbox{\rm Re}\:}
\def\li{\mbox{\rm li}}
\def\ep{\epsilon}
\def\tr{\tilde{R}}
\def\be{\begin{equation}}
\def\ee{\end{equation}}
\def\cn{{\mathcal N}}
\def\sn{{\mathbb  S}^{n-1}}
\def\Ker {{\rm Ker}\:}
\def\el{E_{\lambda}}
\def\Rc{{\mathcal R}}
\def\Ha{H_0^{ac}}
\def\la{\langle}
\def\ra{\rangle}
\def\Ko{\Ker G_0}
\def\Kd{\Ker G_b}
\def\hc{{\mathcal H}}
\def\caH{\hc}
\def\caO{{\mathcal O}}
\def\la{\langle}
\def\ra{\rangle}
\def\sp{\sigma_{+}}
\def\pa{\partial}
\def\et{|\eta'|^2}
\def\lg{L^2(\Gamma)}
\def\h1{H^1_h(\Gamma)}
\def\fr{\frac}
\def\rn{R_n\Bigl(-\fr{1}{2\lambda}\Bigr)}
\def\drn{R_n'\Bigl(-\fr{1}{2\lambda}\Bigr)}
\def\il{-\ii \lambda}
\def\2l{-\fr{1}{2\lambda}}
\def\dd{\fr{d}{d\lambda}}
\def\oc{{\mathcal O}}
\def\hn{h^{(1)}_n}
\def\hnn{h^{(1)}_{n+1}}
\title[Weyl formula] {Weyl formula for the negative dissipative eigenvalues of Maxwell's equations}

\author[F. Colombini] {Ferruccio Colombini}
\address{Dipartimento di Matematica, Universit\`a di Pisa, Italia}
\email{colombini@dm.unipi.it}
\author[V. Petkov]{Vesselin Petkov}
\address{Institut de Math\'ematiques de Bordeaux, 351,
Cours de la Lib\'eration, 33405  Talence, France}
\email{petkov@math.u-bordeaux.fr}


\begin{abstract}
 Let $V(t) = e^{tG_b},\: t \geq 0,$ be the semigroup
  generated by Maxwell's equations  in an exterior domain $\Omega \subset \R^3$  with dissipative boundary condition $E_{tan}- \gamma(x) (\nu \wedge B_{tan}) = 0, \gamma(x) > 0, \forall x \in \Gamma = \pa \Omega.$  We study the case when $\Omega = \{x \in \R^3:\: |x| > 1\}$ and $\gamma \neq 1$ is a constant. We establish a Weyl formula for the counting function of the negative real eigenvalues of $G_b.$

\end{abstract}

\maketitle

\section{Introduction}

Let $K \subset \{ x\in \R^3: \: |x| \leq a\}$ be an open connected domain 
and let $\Omega = \R^3 \setminus \bar{K}$  
be connected  domain with $C^{\infty}$ smooth boundary $\Gamma$.
Consider the boundary problem
\begin{equation}  \label{eq:1.1}
\begin{aligned} 
&\partial_t E = \curl B,\qquad \partial_t B = -\curl E \quad {\rm in}\quad \R_t^+ \times \Omega,
\\
&E_{tan} - \gamma(x)(\nu \wedge B_{tan}) = 0 \quad{\rm on} \quad \R_t^+ \times \Gamma,
\\
&E(0, x) = E_0(x), \qquad B(0, x) = B_0(x). 
\end{aligned}
\end{equation}
with initial data $f = (E_0, B_0) \in  (L^2(\Omega))^6 = {\mathcal H}.$
Here $\nu(x)$ is the unit outward normal to $\partial \Omega$ at $x \in \Gamma$ pointing into $\Omega$, 
$\la\: , \ra$ denotes the scalar product in $\C^3$,
$u_{tan} := u - \la u, \nu\ra \nu$, and $\gamma(x) \in C^{\infty}(\Gamma)$ satisfies $\gamma(x) > 0$ for all $ x \in \Gamma.$ The solution of the problem (\ref{eq:1.1}) is described  by a contraction semigroup 
$$(E, B) = V(t)f = e^{tG_b} f,\: t \geq 0,$$
 where the generator
$G_b$ has  domain $D(G_b)$
which is the closure in the graph norm of functions $u = (v, w) \in (C_{(0)}^{\infty} (\R^3))^3 \times (C_{(0)}^{\infty} (\R^3))^3$ satisfying the boundary condition $v_{tan} - \gamma (\nu \wedge w_{tan}) = 0$ on $\Gamma.$ \\

In \cite{CPR1} it was proved that the spectrum of $G_b$ in the open half plan $\{ z \in \C:\: \re z < 0\}$
is formed by isolated eigenvalues with finite multiplicities. Note that if $G_b f = \lambda f$ with $\re \lambda < 0$, the solution $u(t, x) = V(t) f = e^{\lambda t} f(x) $ of (\ref{eq:1.1}) has exponentially decreasing global energy. Such solutions are called {\bf asymptotically disappearing} and they are very important for the inverse scattering problems (see \cite{CPR1}). In particular, the eigenvalues $\lambda$ with $\re \lambda \to -\infty$ imply  a very fast decay of the corresponding solutions. In \cite{CPR2} the existence of eigenvalues of $G_b$ has been studied for the ball $B_3 = \{x\in \R^3,\:|x| < 1\}$ assuming $\gamma$ constant. It was proved for $\gamma = 1$ there are no eigenvalues in $\{z \in \C: \re z < 0\}$, while for $\gamma \neq 1$ there is always an infinite number of real eigenvalues $\lambda_m$ and with exception of one they satisfy the estimate 
\begin{equation} \label{eq:1.2}
\lambda_m  
\ \leq\
 - \frac{1}{\max \{ (\gamma_0 - 1), \sqrt{\gamma_0 - 1}\}} = -c_0
 \,,
\end{equation}
where $\gamma_0 = \max\{\gamma, \frac{1}{\gamma}\}.$\\

In this Note we study the distribution of the negative eigenvalues and our purpose is to obtain a Weyl formula for the counting function
$$N(r) = \#\{\lambda \in \sigma_p(G_b) \cap \R^-: \: |\lambda| \leq r\},\: r > r_0(\gamma),$$
where every eigenvalues $\lambda_m$ is counted with its algebraic multiplicity given by
$${\rm mult}(\lambda_m)= {\rm rank}\: \frac{1}{2 \pi \ii} \int_{|\lambda_n - z|= \epsilon} (z - G_b)^{-1} dz,$$
where $0 < \epsilon \ll 1$. Our main result is the following

\begin{thm}
Let $\gamma \neq 1$ be a constant and let $\gamma_0 = \max \{\gamma, \frac{1}{\gamma}\}.$ Then the counting function $N(r)$ for the ball $B_3$ has the asymptotic
\begin{equation} \label{eq:1.3}
N(r) = (\gamma_0^2 - 1) r^2 + \oc_{\gamma}(r),\: r \geq r_0(\gamma) > c_0.
\end{equation}
\end{thm}
 The proof of Theorem 1.1 is based on a precise analysis of the roots of the equation (\ref{eq:3.1}) involving spherical Hankel functions $\hn(\lambda)$ of first kind.  We show in Section 3 that for $\gamma > 1$ this equation has only one real root $\lambda_n < 0$. Moreover, we have $\lambda_{n+1} < \lambda_n,\: \forall n\in \N,$ so we have a decreasing sequence of eigenvalues. The geometric multiplicity of $\lambda_n$ is $2n + 1$.  Since $C_b$ is not a self-adjoint operator the geometric multiplicity could be less than the algebraic one. In our case these multiplicities coincide and the proof is based on a representation of $(G_b - z)^{-1}.$ To estimate $\lambda_n$ as $n \to \infty$, we apply an approximation of the exterior semiclassical Dirichlet to Neumann map for the operator $(h^2 \Delta + z)$ established in \cite{P} (see also \cite{V}) combined with an application of Rouch\'e theorem.\\

We conjecture that in the general case of strictly convex obstacles and $\min_{y \in \Gamma} \gamma(y) = \gamma_1 > 1$ we  have the asymptotic
$$N(r) = \frac{1}{4 \pi} \Bigl(\int_{\Gamma} (\gamma^2(y) - 1) dS_y\Bigr) r^{2} + \oc_{\gamma}(r),\: r \geq r_0(\gamma_0).$$

For the ball $B_3$ this agrees with (\ref{eq:1.3}).

\section{Boundary problem for Maxwell system}
\def\hn{h^{(1)}_n}

Our purpose is to study  the  eigenvalues of $G_b$ in case the obstacle is equal to the 
 ball $B_3 = \{x \in \R^3: |x| \leq 1\}$. Setting $\lambda = \ii \mu$, $\im \mu > 0,$ 
 an eigenfunction $(E, B) \neq 0$ of $G_b$ 
 satisfies
\begin{equation} \label{eq:2.1}
\curl E = - \ii \mu B,\qquad  \curl B = \ii \mu E.
\end{equation}
Replacing $B$ by $H = - B$ yields for $(E, H) \in (H^2(|x| \leq 1))^6$,
\begin{equation} \label{eq:2.2}
\begin{cases}\curl E =  \ii \mu H,\qquad  \curl H = -\ii \mu E,\quad {\rm for}\quad x\in B_3,  \\
E_{tan} + \gamma (\nu \wedge H_{tan})= 0,\quad {\rm for }\quad x \in \S^2.
\end{cases}
\end{equation}
Expand $E(x), H(x)$ in
 the spherical functions $Y_n^m(\omega), \: n = 0, 1, 2,..., \: |m| \leq n,\: \omega \in \S^2$ and the spherical Hankel functions of first kind
$$h^{(1)}_n(z): = \fr{H^{(1)}_{n+ 1/2}(z)}{\sqrt z}, n \geq 1$$

An application of Theorem 2.50 in \cite{KH} (in the notation of \cite{KH} it is necessary to replace $\omega$ by $\mu \in \C \setminus \{0\}$) says that the solution of the system (\ref{eq:2.2}) for $x =|x| \omega, r = |x|> 0, \omega = \frac{x}{r}$ has the form
\begin{eqnarray}\label{eq:2.3}
E(x) =  \sum_{n=1}^{\infty}\sum_{|m|\leq n}\Bigl[\alpha_n^m \sqrt{n (n+1)} \frac{\hn(\mu r)}{r} Y_n^m(\omega) \omega \nonumber\\
+ \frac{\alpha_n^m}{r} (r \hn(\mu r))' U_n^m (\omega) + \beta_n^m \hn(\mu) V_n^m(\omega)\Bigr],
 \end{eqnarray}
\begin{eqnarray}\label{eq:2.4}
H(x) = -\frac{1}{\ii \mu} \sum_{n=1}^{\infty}\sum_{|m|\leq n} \Bigl[\beta_n^m \sqrt{n(n+1)}\frac{\hn(\mu r)}{r} Y_n^m(\omega)\omega \nonumber \\
+ \frac{\beta_n^m}{r} (r \hn(\mu r))' U_n^m (\omega) + \mu^2\alpha_n^m \hn(\mu) V_n^m(\omega)\Bigr].
\end{eqnarray}
Here $U_n^m (\omega) = \frac{1}{\sqrt{n(n+1)}} \grad_{\S^2} Y_n^m(\omega)$ and $V_n^m(\omega) = \nu \wedge U_n^m(\omega)$ for $ n \in \N, -n \leq m \leq n$ form a complete orthonormal basis in
$$L^2_t(\S^2) = \{ u(\omega) \in (L^2(\S^2))^3:\: \la \omega, u(\omega) \ra = 0\: {\rm on}\: \S^2\}.$$
To find a representation of $\nu \wedge H_{tan}$,  
observe that $\nu \wedge (\nu \wedge U_n^m) = - U_n^m,$ so for $r = 1$ one has
$$(\nu \wedge H_{tan})(\omega) =  -\frac{1}{\ii \mu} \sum_{n=1}^{\infty}\sum_{|m|\leq n} \Bigl[\beta_n^m \Bigl(\hn(\mu) + \frac{d}{dr} \hn(\mu r)\vert_{r = 1}\Bigr)V_n^m(\omega)$$
$$ - \mu^2\alpha_n^m \hn(\mu) U_n^m(\omega)\Bigr]$$
and the boundary condition in (\ref{eq:2.2}) is satisfied if
\begin{equation} \label{eq:2.5}
\alpha_n^m \Bigl[ \hn(\mu) + \frac{d}{dr}(\hn(\mu r))\vert_{r = 1} - \gamma \ii \mu \hn(\mu)\Bigr] = 0,\: \forall n \in \N, \: |m|\leq n,
\end{equation}
\begin{equation}\label{eq:2.6}
-\frac{\beta_n^m \gamma}{\ii \mu}\Bigl[  \hn(\mu) + \frac{d}{dr}(\hn(\mu r))\vert_{r = 1} - \frac{\ii \mu}{\gamma} \hn(\mu)\Bigr] = 0,\: \forall n \in \N, \: |m|\leq n.
\end{equation}

\section{Roots of the equation $g_n(\lambda) = 0$}

To examine the eigenvalues of $G_b$ it is necessary to find the roots of the equations (\ref{eq:2.3}) and (\ref{eq:2.4}). Since
$\hn(\mu) \neq 0$ for $\im \mu >0$, the problem is reduced to  study the roots $\lambda \in \R^{-}$ of the equation
\begin{equation} \label{eq:3.1}
1 + \frac{d}{dr} \hn(-\ii \lambda r)\Bigl\vert_{r = 1} (\hn(-\ii\lambda))^{-1} - \lambda \gamma = 0
\end{equation}
and the same equation with $\gamma$ replaced by $\fr{1}{\gamma}$. Clearly, if $\mu = -\ii\lambda$ is such that the expressions in the brackets $[...]$ in (\ref{eq:2.5}) and (\ref{eq:2.6}) are non-vanishing for every $n \geq 1$, we must have $\alpha_n^m = \beta_n^m = 0$ which implies $E_{tan} = B_{tan} = 0$. Hence $(E, B) = 0$ because the boundary problem with $\gamma = 0$ has no eigenvalues in $\{z\in \C:\:\re z < 0\}.$ 
In this section we suppose that $\gamma \neq 1$ and examine the equation
\begin{equation} \label{eq:3.2}
g_n(\lambda): = \frac{1}{\lambda} + \dd \Bigl( \hn(-\ii\lambda)\Bigr)(\hn(-\ii \lambda))^{-1} - \gamma = 0.
\end{equation}
It is well known that (see \cite{O}) 
$$\hn(-\ii \lambda) = (-\ii)^{n+1}\fr{e^{\lambda}}{-\ii \lambda}R_n\Bigl(\fr{\ii}{-2 \ii \lambda}\Bigr) = (-\ii)^{n}\fr{e^{\lambda}}{\lambda}R_n\Bigl(-\fr{1}{2\lambda}\Bigr)$$
with 
$$R_n(z): = \sum_{m=0}^n a_{m,n} z^m,\: a_{m,n} = \frac{(n + m)!}{m! (n-m)!} > 0.$$
We will prove the following
\begin{prop}

For $\lambda < 0$ we have
\begin{equation}
G_{n, n+1}(\lambda) =\fr {\fr{d}{d\lambda} \hnn(-\ii \lambda)}{\hnn(-\ii \lambda)} - \fr {\fr{d}{d\lambda} \hn(-\ii \lambda)}{\hn(-\ii \lambda)}> 0.
\end{equation}

\end{prop}

\begin{proof} The purpose is to show that
$$\Bigl(\hn(\il) \dd \hnn(\il) - \hnn(\il) \dd \hn(\il)\Bigr)\Bigl(\hnn(\il) \hn(\il)\Bigr)^{-1} > 0.$$
Introduce the functions
$$\xi_n(\lambda): = \fr{e^{\lambda}}{\lambda} R_n\Bigl(\2l\Bigr),\: \eta_n(\lambda): =\lambda \xi_n(\lambda).$$
Then
$\hn(\il) = (-\ii)^n \xi_n(\lambda)$ and the above inequality is equivalent to
$$\Bigl(\xi_n(\lambda) \dd \xi_{n+1}(\lambda) - \xi_{n+1}(\lambda) \dd \xi_n(\lambda)\Bigr)\Bigl(\xi_{n+1}(\lambda) \xi_n(\lambda)\Bigr)^{-1}$$
$$=\Bigl(\eta_n(\lambda) \dd \eta_{n+1}(\lambda) - \eta_{n+1}(\lambda) \dd \eta_n(\lambda)\Bigr)\Bigl(\eta_{n+1}(\lambda) \eta_n(\lambda)\Bigr)^{-1} > 0.$$
Since $\eta_n(\lambda) \eta_{n+1}(\lambda) > 0$ for $\lambda < 0$, it suffices to show that the function
$$F(\lambda) = \eta_n(\lambda) \dd \eta_{n+1}(\lambda) - \eta_{n+1}(\lambda) \dd \eta_n(\lambda)$$
has  positive values for $\lambda \in (-\infty, 0).$
Consider the derivative
$$F'(\lambda) = \eta_n(\lambda) \fr{d^2}{d\lambda^2} \eta_{n+1}(\lambda) - \eta_{n+1}(\lambda) \fr{d^2}{d\lambda^2} \eta_n(\lambda).$$
We have
$$\eta_n(\lambda) =  \ii^{n+1} \hn(\il) (\il) = \ii^{n+1} \Xi_n(\il)= -\ii^{n-1} \Xi_n(\il).$$
The function $\Xi_n(z)= z\hn(z)$ satisfies the equation
$$\Xi_n''(z) + \Bigl(1 - \frac{n^2 + n}{z^2}\Bigr) \Xi_n(z) = 0$$
and
$$\fr{d^2}{d\lambda^2} \eta_n(\lambda) = \ii^{n-1} \Xi_n''(\il) = -\ii^{n-1}\Bigl( 1 + \fr{n^2 +n}{\lambda^2}\Bigr) \Xi_n(\il)$$
$$= \Bigl( 1 + \fr{n^2 +n}{\lambda^2}\Bigr)\eta_n(\lambda).$$
Consequently,
$$F'(\lambda) = \Bigl[\fr{(n+1)^2 + n+1}{\lambda^2} - \fr{n^2 + n}{\lambda^2}\Bigr] \eta_n(\lambda) \eta_{n+1}(\lambda)$$
$$= 2(n + 2)\fr{\eta_n(\lambda) \eta_{n+1}(\lambda)}{\lambda^2} > 0.$$

 On the other hand,
$$F(\lambda) = e^{\lambda} \rn \dd \Bigl(e^{\lambda} R_{n+1}\Bigl(\2l\Bigr)\Bigr) - e^{\lambda} R_{n+1}\Bigl(\2l\Bigr) \dd \Bigl(e^{\lambda} \rn\Bigr)$$
$$= \fr{e^{2\lambda}}{2 \lambda^2}\Bigl[ \rn R_{n+1}'\Bigl(\2l\Bigr) - R_{n+1}\Bigl(\2l\Bigr) R_n'(\2l)\Bigr]$$
and 
$$\lim_{\lambda \to - \infty} F(\lambda) = 0,\: \lim_{\lambda \nearrow 0} F(\lambda)= +\infty$$
since
$$\lim_{w \to +\infty} \Bigl[R_n(w) R_{n+1}'(w) - R_{n+1}(w) R_n'(w)\Bigr] = +\infty.$$
Finally, the function $F(\lambda)$ in the interval $(-\infty, 0]$ is increasing from 0 to $+\infty$ and this completes the proof.
\end{proof}

Now if $\lambda_n < 0$ is a solution the equation
\begin{equation} \label{eq:3.4}
g_n(\lambda) : = \frac{1}{\lambda} + \Bigl( \frac{d}{d \lambda} \hn(-\ii \lambda)\Bigr) (\hn(-\ii \lambda))^{-1} - \gamma = 0,
\end{equation}
one has 
$$g_{n+1}(\lambda_n) = \frac{1}{\lambda_n} + \Bigl(\dd \hnn(-i\lambda_n)\Bigr) (\hnn(-i\lambda_n))^{-1} - \gamma = G_{n, n+1}(\lambda_n)  > 0,$$
so $\lambda_n$ is not a root of the equation 
$$g_{n+1}(\lambda) = \frac{1}{\lambda} + \Bigl(\frac{d}{d\lambda} \hnn(-\ii \lambda) \Bigr) (\hnn(-\ii \lambda))^{-1} - \gamma = 0.$$

In the following we assume that $\gamma > 1.$ Then for $\lambda \to -\infty$ we have $g_{n+1}(\lambda) \to 1- \gamma < 0$, and since $g_{n+1}(\lambda_n) > 0$ the equation $g_{n+1}(\lambda) = 0$ has at least one root $-\infty < \lambda_{n+1} <\lambda_n.$

\begin{lem} Let $\gamma > 1$. For every $n \geq 1$ the equation $g_n(\lambda) = 0$ in the interval $(-\infty, 0)$ has exactly one root $\lambda_{n} < 0.$
\end{lem}
\begin{proof} Setting $ w = -\fr{1}{2\lambda}$, we write the equation (\ref{eq:3.2}) as  ${\mathcal R}_n(w): = w^2 R_n'(w) + \alpha R_n(w) = 0$, where $\alpha = \fr{1- \gamma}{2} < 0.$
We will show that this equation has exactly one positive root.
Since 
$$w^2R_n'(w) = \sum_{k=1}^{n} k a_{k,n} w^{k+1},\: R_n(w) = \sum_{k=0}^n a_{k,n} w^{k},$$
the polynomial ${\mathcal R}_n(w)$ has the representation
$${\mathcal R}_n(w) = \sum_{k=0}^{n+1} b_{k, n} w^k$$
 with
$$\begin{cases}  b_{k, n} = (k-1)a_{k-1,n} + \alpha a_{k, n}, \: 0\leq k \leq n, \: a_{-1, n} = 0,\\
b_{n + 1, n} = \fr{(2n)!}{(n - 1)!}.\end{cases}$$
Taking into account the form of $a_{k, n}$, we deduce
\begin{equation} \label{eq:3.4}
b_{k, n} = \fr{(n+ k - 1)!}{(n-k + 1)! k!}\Bigl (k(k-1) + \alpha(n+k)(n-k+1)\Bigr),\: 0 \leq k \leq n+1.
\end{equation}
Thus the sign of $b_{k, n}$ depends on the sign of the function 
$$B(k): = (1-\alpha)k^2 + (\alpha-1) k + \alpha(n^2 + n)$$
which for $k \geq 1$ is increasing since 
$$B'(k) = 2(1-\alpha)k + \alpha-1 \geq 1 -\alpha > 0.$$
Clearly, $b_{0, n} = \alpha < 0$ and $b_{n+1, n} > 0.$ There are two cases:\\

(i) $b_{1, n} \leq 0.$ Then there is only one change of sing in the Descartes' sequence $\{b_{n+1, n}, b_{n,n},...,b_{1, n}, b_{0, n}\}.$\\

(ii) $b_{1, n} > 0$. Then $b_{k, n} > 0$ for $1 \leq k \leq n+1$ and in the Descartes' sequence $\{b_{n+1,n}, b_{n,n},...,b_{1,n}, b_{0, n}\}$ one has again only one change of sign.

Applying the Descartes' rule of signs, we conclude that the number of the positive roots of ${\mathcal R}_n(w) = 0$ is exactly one.

\end{proof}
Combining Proposition 3.1 and Lemma 3.1, one obtain the following

\begin{cor} Let $\gamma > 1$. Then the generator $G_b$ has an infinite sequence of real eigenvalues
$$-\infty < ... < \lambda_n <...<\lambda_2 < \lambda_1 < 0$$
and $\lambda_n$ has geometric multiplicity $2n + 1.$
\end{cor}
The last statement concerns the geometric multiplicity since the functions $\{Y_{m, n}(\omega)\}_{m=-n}^m$ are linearly independent. The algebraic multiplicity of $\lambda_m$ will be discussed in Section 5.

\section{Estimation of the roots}
\def\il{\ii \lambda}
Throughout this section we assume $\gamma > 1$. Set $\lambda = \frac{\ii \sqrt{z}}{h},\: 0 < h \ll 1$ with
$z = - 1 + \ii\eta, \: 0 \leq |\eta| \leq h^{1/2}, \eta \in \R.$ Consider the Dirichlet problem
\begin{equation} \begin{cases} \label{eq:4.1}
(h^2\Delta  + z) w = 0, \: |x| >1, w \in H^2(|x| > 1),\\
w = f,\: |x| = 1
\end{cases} 
\end{equation}
and note that $\Delta + \frac{z}{h^2} = \Delta - \lambda^2.$
The solution of (\ref{eq:4.1}) has the form 
$$w (r \omega)= \sum_{n=0}^{\infty} \sum_{m = -n}^n \hn(-\il r) (\hn(-\il)^{-1}\alpha_{n, m} Y_{n, m}(\omega),$$
where
$$ f(\omega) =  \sum_{n=0}^{\infty}\sum_{m = -n}^n \alpha_{n, m} Y_{n, m}(\omega).$$
The semiclassical Dirichlet-to-Neumann operator ${\mathcal N}_{ext}(h, z) =\fr{h}{\ii} \fr{d}{dr} w\vert_{r=1}$ related to
(\ref{eq:4.1}) becomes

$${\mathcal N}_{ext}(h, z) = -\ii \sqrt{z}\sum_{n=0} ^{\infty}\sum_{m = -n}^n  (\hn)' (-\il ) (\hn(-\il))^{-1} \alpha_{n, m} Y_{n, m}$$
$$ = \sqrt{z} \sum_{n=0} ^{\infty}\sum_{m = -n}^n \dd \Bigl(\hn(-\il)\Bigr) (\hn(-\il))^{-1} \alpha_{n, m} Y_{n, m}.$$
By using the approximation of ${\mathcal N}_{ext}(h, z)$ established in \cite{V},\cite{P} for $ z = - 1 + \ii \eta$, one deduces
$$\|{\mathcal N}_{ext}(h, z)f - Op_h(\rho)f\|_{L^2(\S^2)} \leq C \frac{|\sqrt{z}|}{|\lambda|}\|f\|_{L^2(\S^2)}, \: 0 < h \leq h_0$$
with $\rho = \sqrt{z- r_0(x', \xi')}$ and a constant $C > 0$ independent of $z, \lambda$ and $f$. Here $r_0(x', \xi')$ is the principal symbol of the semiclasssical Laplace-Beltrami operator $-h^2\Delta_{\S^2}= \frac{z}{\lambda^2} \Delta_{\S^2}$. Moreover, $\sqrt{z} = \ii\sqrt{1 - \ii \eta} = \ii(1 -\frac{\ii \eta}{2} + \oc(\eta^2))$ and
$$ \re \lambda = -\frac{1}{h} + \oc(1),\: \im \lambda = \oc(h^{-1/2}).$$
Hence, for $0 < h \leq h_0$ we get
$$ \lambda \in \Lambda_0 =\{ z\in \C: \: |\im z| \leq ch_0^{1/2}|\re z|, \: \re \lambda < -\epsilon < 0,\:\:|\lambda| \geq \lambda_0\}.$$

On the other hand,
$$\Bigl\|Op_h(\rho)  - \sqrt{z}\Bigl( \sqrt{ 1 - \fr{\Delta_{\S^2}}{\lambda^2}}\Bigr)\Bigr\|_{L^2(\S^2)\to L^2(\S^2)} \leq C_1 |\lambda|^{-1}, \: \lambda \in \Lambda_0.$$
 Applying the spectral theorem, one deduces
$$\Bigl(\sqrt{ 1 -\fr{\Delta_{S^2}}{\lambda^2}}\Bigr)f = \sum_{n=0}^{\infty}\sum_{m = -n}^n \Bigl(\sqrt{1 +\fr{n(n+1)}{\lambda^2}}\Bigr) \alpha_{n,m} Y_{n. m}$$
and
$$\Bigl\|\Bigl({\mathcal N}_{ext}(h, -z) - \sqrt{z}\Bigl(\sqrt{ 1 - \fr{\Delta_{S^2}}{\lambda^2}} \Bigr)f\Bigr\|^2_{L^2(S^2)}  = |z|\sum_{n=0}^{\infty} \sum_{m=-n}^n \Bigl| \dd \Bigl(h_n(-i \lambda )\Bigr) (h_n(-i\lambda))^{-1}$$ $$- \sqrt{1 +\fr{n(n+1)}{\lambda^2}}\Bigr|^2 |a_{n, m}|^2.$$
This implies 
\begin{equation} \label{eq:4.2}
\Bigl| \dd \Bigl(\hn(-\il )\Bigr) (\hn(-\il))^{-1}- \sqrt{1 +\fr{n(n+1)}{\lambda^2}}\Bigr| \leq C_2 |\lambda|^{-1}, \: \forall n \in \N, \: \lambda \in \Lambda_0
\end{equation}
which we write as
\begin{equation} \label{eq:4.3}
\Bigl|\Bigl[ \frac{1}{\lambda} + \dd \Bigl(\hn(-\il)\Bigr) (\hn(-\il))^{-1}- \gamma\Bigr] -\Bigl[ \sqrt{1 +\fr{n(n+1)}{\lambda^2}}- \gamma\Bigr]\Bigr| \leq C_0|\lambda|^{-1}.
\end{equation}
\begin{rem} For bounded $1 \leq n \leq N_0$ and sufficiently large $|\lambda|$ the estimate $(4.2)$ follows easily from the fact that $\frac{R_n'(w)}{R_n(w)} = n(n+1) + \oc(|w|)$ as $|w| \to 0.$
\end{rem}
\begin{rem}  The estimate $(4.2)$ is similar to that in Proposition 2.1 in \cite{PV}, where the function $\frac{J_{\nu}'(\lambda)}{J_{\nu}(\lambda)}$ for $\nu \geq 0$ and $0 < C \leq |\im \lambda| \leq \delta |\re \lambda|, \: \re \lambda > C_1$ has been approximated. Here $J_{\nu}(z)$ is the Bessel function, while the boundary problem examined in \cite{PV} is in the domain $|x| < 1.$ 
\end{rem}
Put $ z = \lambda$ and for $z \in \Lambda_0$ consider the function
$$f_n(z) : = \sqrt{1 +\fr{n(n+1)}{z^2}}-  \gamma $$
with zeros
$$z_{n}^{\pm} = \pm  \sqrt{\frac{n^2 + n}{\gamma^2 - 1}}.$$
In the following we set $z_n = -\sqrt{\fr{n(n+1)}{\gamma^2 -1}}.$ Clearly,
$$f_n'(z) = - \fr{1}{z} \fr{ \fr{n(n+1)}{z^2}}{\sqrt{1 +\fr{n(n+1)}{z^2}}}$$
and
$\fr{n(n+1)}{z_n^2} = \gamma^2 - 1,\: f_n'(z_n) = -\fr{\gamma^2 - 1}{ \gamma z_n}.$
A calculus yields the second derivative
$$f_n''(z) = \fr{1}{z^2}\Bigl[ \fr{3 n(n+1)}{z^2}\Bigl(\sqrt{1 + \fr{n(n+1)}{z^2}}\Bigr) $$
$$- \fr{n^2(n+1)^2}{z^4} \Bigl(\sqrt{1 + \fr{n(n+1)}{z^2}}\Bigr)^{-1/2}\Bigr] \Bigr( 1 +\fr{n(n+1)}{z^2}\Bigr)^{-1}.$$
For $n$ large enough and $a > 0$ to be fixed below introduce the contour 
$$C_n(a):= \{z = z_n  + a e^{\ii\varphi},\: 0 \leq \varphi < 2\pi\} \subset \Lambda_0.$$

 Our purpose is to choose $a$ so  that
\begin{equation} \label{eq:4.4}
|f_n(z)| \geq \fr{C_0}{|z|},\:\: \forall z \in C_n(a).
\end{equation}

We have 
$$z^2 = z_n^2 + 2 z_n a e^{\ii \varphi} + a^2 e^{2\ii\varphi}$$
and
\begin{equation} \label{eq:4.5}
\fr{n(n+1)}{z^2} = (\gamma^2 - 1)\Bigl( 1 + \oc \Bigl(\fr{1}{n}\Bigr) a+ \oc\Bigl(\fr{1}{n^2}\Bigr) a^2 \Bigr)^{-1}, \: z \in C_n(a).
\end{equation}
On the other hand,
$$\sqrt{ \fr{n(n+1)}{z^2} +1} =  \Bigl[\fr{\gamma^2 +\oc \Bigl(\fr{1}{n}\Bigr) a + \oc(\fr{1}{n^2})a^2}{1 + \oc \Bigl(\fr{1}{n}\Bigr) a + \oc(\fr{1}{n^2}) a^2}\Bigr]^{1/2}.$$
Clearly, one has the estimate
\begin{equation} \label{eq:4.6}
|f_n(z)| \geq   \fr{\gamma^2 - 1}{\gamma |z_n|}a  - \fr{a^2}{2} \sup_{z\in C_n(a)} |f_n''(z)|,\: z \in C_n(a).
\end{equation}
Set $C_{\gamma} = \fr{\gamma^2 - 1}{\gamma} > 0$ and  choose $a > 0$ so that $C_{\gamma} a > 4 C_0$. We fix $a$ and obtain
$$\fr{C_{\gamma}a}{2 |z_n|} > \fr{2C_0}{|z_n|} >   \fr{C_0}{|z_n| | 1 +\fr{a e^{\ii \varphi}}{z_n}|},\: 0 \leq \varphi < 2 \pi,$$
taking $n$ large enough to satisfy the inequality
$$ \fr{1}{\Bigl| 1 +\fr{a e^{\ii \varphi}}{z_n} \Bigr|} < 2.$$
Next we  arrange the inequality
\begin{equation} \label{eq:4.7}
\fr{C_{\gamma} a}{2 |z_n|} - \fr{a^2}{2} \sup_{z\in C_n(a)} |f_n''(z)| > 0.
\end{equation}
It is clear that 
$$f_n''(z) = \fr{1}{z^2} G\Bigl(\fr{n(n+1)}{z^2}\Bigr),$$
where 
$$G(\zeta) = \Bigl[3 \zeta \sqrt{\zeta + 1} - \zeta^2 (\zeta + 1)^{-1/2}\Bigr](\zeta + 1)^{-1}.$$
Note that for $z \in C_n(a)$ and $n$ large enough according to (\ref{eq:4.4}), the function $|G(\fr{n(n+1)}{z^2})|$ is bounded by a constant $B_{\gamma, a}$ depending on $\gamma$ and $a$
. Thus for large $n$ we get
$$\sup_{z \in C_n(a)} |f_n''(z)| \leq B_{\gamma, a} \sup_{z \in C_n(a)}\fr{1}{|z|^2} = B_{\gamma, a} \fr{1}{|z_n|^2} \sup_{z \in C_n(a)} \fr{1}{|1 + \fr{a e^{i\varphi}}{z_n}|^2}\leq 4 B_{\gamma, a}\fr{1}{|z_n|^2}$$
and the proof of (\ref{eq:4.7}) is reduced to
$$C_{\gamma} > 4 B_{\gamma, a}\fr{a}{|z_n|}$$
which is satisfied taking again $n$ large. Finally, we proved the estimate (\ref{eq:4.3}) and we can apply Rouch\'e theorem for the functions $g_n(z)$ and $f_n(z)$ and conclude that the function
$g_n(z)$
has exactly one simple zero $\lambda_n$ in $C_n(a)$. Since $g_n(z)$ has only real zeros (see Appendix in \cite{CPR2}), this implies the following

\begin{lem} There exist $n_0(\gamma)$ and $a(\gamma) > 0$ depending on $\gamma$ such that for $n \geq n_0(\gamma)$ the negative root $\lambda_n$ of the equation (\ref{eq:3.2}) satisfies the estimate
\begin{equation} \label{eq:4.8}
\Bigl| \lambda_n + \sqrt{\fr{n(n+1)}{\gamma^2 - 1}}\Bigr | \leq a(\gamma).
\end{equation}
\end{lem}
\begin{rem} According to Proposition 2.1, $n_0(\gamma)$ must satisfy the inequality
$$n_0(\gamma) \geq \frac{\sqrt{\gamma^2 - 1}}{\max\{\gamma - 1, \sqrt{\gamma - 1}\}}.$$
\end{rem}

\section{Weyl asymptotics}

We start with the analysis of the multiplicity of $\lambda_n.$
\begin{lem} For $n \geq n_0(\gamma)$ we have ${\rm mult}(\lambda_n) = 2n + 1.$
\end{lem}
\begin{proof} Since the geometric multiplicity of $\lambda_n$ is $2n +1$, it is sufficient to show that
\begin{equation} \label{eq:5.1}
{\rm mult}(\lambda_n) \leq 2n + 1.
\end{equation}
Let $\lambda \in \Lambda_0,$ where $\Lambda_0$ is the set introduced in the previous section and let $ \lambda \notin \sigma(G_b)$. If $0 \neq (f, g) \in {\rm Image}\: G_b \cap L^2(\Omega)$, one has $\dive f = \dive g = 0$ and for
$(u, v) = (G_b - \lambda)^{-1} (f, g)$ we get $\dive u = \dive v = 0.$
Consider the skew self-adjoint operator 
$$A = \begin{pmatrix} 0 & -\curl \\ \curl & 0 \end{pmatrix}$$
with boundary condition $\nu \wedge u = 0$ on $\S^2.$ Then $\sigma(A) \subset \ii \R$ and let $(u_0(x; \lambda), v_0(x; \lambda)) = (A- \lambda)^{-1}(f, g),$ that is 

\begin{equation} \label{eq:5.2}
\begin{cases} (A - \lambda) \begin{pmatrix} u_0\\ v_0 \end{pmatrix} = \begin{pmatrix} f\\ g  \end{pmatrix} \: {\rm for}\:|x| >1,\\
\nu \wedge u_0= 0 \: {\rm on}\: \S^2. \end{cases}
\end{equation} 
Since $\dive u_0 =\dive v_0 = 0$, the well known coercive estimates yield $(u_0, v_0) \in H^1(\Omega)$. Moreover the resolvent
$(A - \lambda)^{-1}$ is analytic in $\{z \in \C:\:\Re z < 0\}$ and $u_0(x; \lambda),\: v_0(x; \lambda)$ depend analytically on $\lambda$.
We write $(u, v) = (u_0, v_0) + (u_1, v_1)$, where $(u_1(x; \lambda), v_1(x; \lambda))$ is the solution of the problem
\begin{equation}\label{eq:5.3}
\begin{cases} (G - \lambda) \begin{pmatrix} u_1\\ v_1 \end{pmatrix} = \begin{pmatrix} 0\\ 0  \end{pmatrix} \: {\rm for}\:|x| >1,\\
 (u_1)_{tan} - \gamma (\nu \wedge (v_1)_{tan}) = - \gamma (\nu \wedge (v_0)_{tan}(x; z))\: {\rm on}\: \S^2. \end{cases}
\end{equation}
To solve (\ref{eq:5.3}), note that $-\gamma(\nu \wedge (v_0)_{tan}(\omega; z)) = F(\omega; \lambda) \in L^2(\S^2)$ with
$F(\omega; \lambda)$ analytical in $\lambda$ for $\lambda \in \Lambda_0$. Thus we may write
$$F(\omega; \lambda) = \sum_{n=1}^{\infty} \sum_{m = -n}^n \tilde{\alpha}_n^m (\lambda) U_n^m(\omega) + \tilde{\beta}_n^m(\lambda) V_n^m(\omega)$$
with analytical coefficients $\tilde{\alpha}_n^m(\lambda), \tilde{\beta}_n^m(\lambda)$. Now we can solve (\ref{eq:2.5}), (\ref{eq:2.6}) with right hand part $(\tilde{\alpha}_n^m(\lambda),\: \tilde{\beta}_n^m(\lambda))$. Finally, we obtain a representation of the solution of (\ref{eq:5.3}) with meromorphic coefficients
$$\alpha_n^m(\lambda) = \frac{\tilde{\alpha}_n^m(\lambda)}{\hn(-\il)\Bigl[1 +  \frac{d}{dr}(\hn(-\il r))\vert_{r = 1}(\hn(-\il))^{-1} - \lambda\gamma \Bigr]},$$
$$ \beta_n^m(\lambda) = -\frac{\lambda\tilde{\beta}_n^m(\lambda)} {\gamma \hn(-\il)\Bigl[1 +  \frac{d}{dr}(\hn(-\il r))\vert_{r = 1}(\hn(-\il))^{-1} - \lambda\gamma^{-1} \Bigr]}.$$
 If $\gamma > 1$ the analysis in the previous section shows that for $\lambda \in \Lambda_0$
the meromorphic function $\alpha_n^m(\lambda)$ has a simple pole at $\lambda_n < 0$, while $\beta_n^m(\lambda)$ is analytic in $\Lambda_0.$ For $0 <\gamma < 1$ the function $\alpha_n^m(\lambda)$ is analytic in $\Lambda_0$ and $\beta_n^m(\lambda)$ is meromorphic.  Next we integrate $(u(x; \lambda), v(x; \lambda))$ over  the circle
$|\lambda_n - \lambda| = \ep$, where $\ep$ is sufficiently small. The integral of $(u_0(x;\lambda), v_0(x;\lambda))$ vanish, while for the integral of $(u_1(x; \lambda), v_1(x;\lambda))$, taking into account the representation of the solution of (\ref{eq:5.3}), we will obtain a sum
$$S_n = \begin{cases}
c_n\sum_{m= -n}^m \tilde{\alpha}_n^m(\lambda_n) U_n^m(\omega), \: c_n \neq 0,\: \gamma > 1,\\
d_n \sum_{m= -n}^m \lambda_n\tilde{\beta}_n^m(\lambda_n)\gamma^{-1} V_n^m(\omega),\: d_n \neq 0,\: 0 < \gamma < 1.
\end{cases}$$
This completes the proof of (\ref{eq:5.1}).
\end{proof}
Passing to the analysis of $N(r)$, consider first the case $\gamma >1.$  The root $\lambda_n$ has algebraic multiplicity $2n + 1$ and to find a lower bound of $N(r)$ we apply the estimate
$$|\lambda_n| \leq \sqrt{\fr{n(n+1)}{\gamma^2 -1}} + a(\gamma) < \fr{n+1}{\sqrt{\gamma^2 -1}} + a(\gamma) \leq r$$
for $r \geq a(\gamma) + \fr{n_0(\gamma) + 1}{\sqrt{\gamma^2 - 1}}.$
Then
$$N(r) \geq \sum_{j =n_0(\gamma)}^{[(r - a(\gamma))\sqrt{\gamma^2 -1} - 1]} (2 j + 1) = (\gamma^2 - 1) r^2 + \oc_{\gamma} (r) + A_{\gamma}.$$
To get a upper bound for $N(r),$ we use the estimate
$$|\lambda_n| \geq \sqrt{\fr{n(n+1)}{\gamma^2 - 1}} - a(\gamma) > \fr{n}{\sqrt{\gamma^2 - 1}} - a(\gamma) \geq r$$
for 
$$n \geq (r + a(\gamma))\sqrt{\gamma^2 - 1} \geq 2 a(\gamma) \sqrt{\gamma^2 - 1} + n_0(\gamma) + 1,$$
hence
$$N(r) \leq \sum_{j = n_0(\gamma)}^{[(r + a(\gamma))\sqrt{\gamma^2 - 1}] + 1} (2j + 1) + D_{\gamma}=  (\gamma^2 - 1) r^2 + \oc_{\gamma} (r) + A_{\gamma}'.$$
 If $0 < \gamma < 1$, we have $\fr{1}{\gamma}>1$ and one applies our argument to the  the equation (\ref{eq:2.6}). This completes the proof of theorem 1.1\\

\end{document}